\documentclass[12pt,reqno,sort]{elsarticle}

\usepackage{amsfonts,color,amsmath,amssymb,fancyhdr}
\usepackage{amsthm}
\usepackage{graphicx}

\usepackage[colorlinks,linkcolor=blue, anchorcolor=red,citecolor=green]{hyperref}
\usepackage{subfigure}

\def\Box{\vcenter{\vbox{\hrule\hbox{\vrule
     \vbox to 8.8pt{\hbox to 10pt{}\vfill}\vrule}\hrule}}}

\newcommand{\F}{\mathbb{F}}

\newcommand{\PSL}{\mathrm{PSL}}
\newcommand{\PGL}{\mathrm{PGL}}

\newtheorem{thm}{Theorem}[section]
\newtheorem{lemma}[thm]{Lemma}
\newtheorem{corollary}[thm]{Corollary}

\numberwithin{equation}{section}

\newcommand{\cL}{\mathcal L}

\newcommand{\cP}{\mathcal P}
\newcommand{\cS}{\mathcal S}

\definecolor{Purple}{rgb}{0.5,0,0.5}

\def\SL{{\rm SL}} \def\PSL{{\rm PSL}}  \def\PGL{{\rm
PGL}}

\begin{document}
\newcommand{\stopthm}{\begin{flushright}
		\(\box \;\;\;\;\;\;\;\;\;\; \)
\end{flushright}}
\newcommand{\symfont}{\fam \mathfam}

\title{On finite generalized quadrangles with $\mathrm{PSL}(2,q)$ as an automorphism group}

\date{}
\author[add1]{Tao Feng}\ead{tfeng@zju.edu.cn}
\author[add1]{Jianbing Lu\corref{cor1}}\ead{jianbinglu@zju.edu.cn}\cortext[cor1]{Corresponding author}
\address[add1]{School of Mathematical Sciences, Zhejiang University, Hangzhou 310027, Zhejiang, P.R. China}

\begin{abstract}
Let $\mathcal{S}$ be a finite thick generalized quadrangle, and suppose that $G$ is an automorphism group of $\mathcal{S}$. If $G$ acts primitively on both the points and lines of $\mathcal{S}$, then it is known that $G$ must be almost simple. In this paper, we show that if the socle of $G$ is $\mathrm{PSL}(2,q)$ with $q\geq4$, then $q=9$ and $\mathcal{S}$ is the unique generalized quadrangle of order $2$.
\newline

\noindent\text{Keywords:} generalized quadrangle;  point-primitive; line-primitive; projective linear group; automorphism group	

\end{abstract}	

\maketitle

\section{Introduction}\label{introduction}

A finite \emph{generalized $n$-gon} is a finite point-line incidence geometry whose bipartite incidence graph has diameter $n$ and girth $2n$. It is \emph{thick} if each line contains at least three points and each point is on at least three lines. A finite generalized $3$-gon is simply a projective plane, and a finite generalized $4$-gon is also called a \textit{generalized quadrangle}.  For a thick generalized $n$-gon, there are constants $s,t$ such that each line is incident with $s+1$ points and each point is incident with $t+1$ lines by \cite[Corollary 1.5.3]{GenPoly}, and we say that it has \emph{order} $(s,t)$. The point-line dual of a thick generalized polygon of order $(s,t)$ is a generalized polygon of order $(t,s)$.  An \emph{automorphism} of a generalized $n$-gon $\mathcal{S}$ consists of a permutation of the points and a permutation of the lines which preserve the incidence. An automorphism group of $\mathcal{S}$ is a group of automorphisms, and its full automorphism group is denoted by $\textup{Aut}(\mathcal{S})$.  We refer to the monographs \cite{GenPoly} and \cite{Payne} for more details on generalized polygons and generalized quadrangles.

In 1959 Tits \cite{Tits} introduced the concept of generalized polygons in order to study the simple groups of Lie type systematically, and his work builds a bridge between geometry and group theory.  The Feit-Higman theorem \cite{Feit} shows that finite thick generalized $n$-gons exist only for $n=3,4,6$ or $8$. The automorphism groups of the \emph{classical} generalized polygons are classical or exceptional simple groups of Lie type, and they act primitively on both points and lines. There are many examples of non-classical projective planes and generalized quadrangles, while the construction of non-classical generalized hexagons and octagons is listed as Problem 1 in \cite[Appendix E]{GenPoly}.

There has been extensive work on the classification of finite thick generalized polygons whose automorphism groups satisfy certain transitivity conditions.  Buekenhout and Van Maldeghem \cite{Buekenhout} showed that a finite thick generalized polygon with a point-distance-transitive automorphism group is either classical, dual classical, or the unique generalized quadrangle of order $(3,5)$. Schneider and Van Maldeghem \cite{Schneider} studied the finite thick generalized hexagons and octagons whose automorphism group $G$ is primitive on both the points and lines and transitive on flags. A \emph{flag} is an incident point-line pair. They showed that $G$ must be an almost simple group of Lie type. In \cite{Bamberg2017}, Bamberg et al. strengthened this result by showing that the same conclusion holds under the single assumption of point-primitivity. In \cite{Morgan}, Morgan and Popiel further classified the generalized hexagons and octagons for which the socle of $G$ is a Suzuki group,  or a Ree group of type $^{2}\mathrm{G}_{2}$ or $^{2}\mathrm{F}_{4}$. Recently, Glasby et al. \cite{Glasby}  studied the finite thick generalized hexagons and octagons with a point-primitive automorphism group whose socle is $\mathrm{PSL}(n,q)$, $n\geq2$, and obtained some partial results.

Suppose that $\mathcal{S}$ is a finite thick generalized quadrangle with a point-primitive automorphism group $G$. Then $G$ is not of  holomorph compound type \cite{Bamberg Popiel}. By \cite{Bamberg2021}, $G$ is not an almost simple sporadic group. If $G$ is also transitive on the lines, then the action of $G$ on the points is not of holomorph simple type by \cite{Bamberg Popiel2017} and the generalized quadrangles for which $G$ has a point-regular abelian normal subgroup are classified in \cite{Bamberg2016}. If $G$ is primitive on both points and lines, then it is an almost simple group, and if further $G$ is transitive on flags, then it is almost simple of Lie type, cf. \cite{Bamberg2012}. In this paper, we study the special case where the automorphism group $G$ is primitive on both points and lines and has socle $\mathrm{PSL}(2,q)$, $q\geq4$.  Let $W(q)$ be the generalized quadrangle of order $q$ whose points and lines are the totally singular points and totally singular lines of the symplectic polar space $W(3,q)$ respectively.  We refer to \cite{Payne} for more details about classical generalized quadrangles. The following is our main result.
\begin{thm}\label{main}
Suppose that $G$ is an automorphism group of a finite thick generalized quadrangle $\mathcal{S}$ that is primitive on both points and lines. If $G$ is an almost simple group with socle $\mathrm{PSL}(2,q)$, $q\ge 4$, then $q=9$ and $\mathcal{S}$ is the symplectic quadrangle $W(2)$.
\end{thm}

The example $W(2)$ in Theorem \ref{main} arises due to the isomorphism $\mathrm{PSp}(4,2)'\cong\mathrm{PSL}(2,9)$. This paper is organized as follows. In Section 2, we present some preliminary results on generalized quadrangles and the maximal subgroups of almost simple groups with socle $\mathrm{PSL}(2,q)$, $q\ge 4$. In Section 3, we describe a coset geometry model for finite generalized quadrangles with an automorphism group that is  transitive on both points and lines. In Section 4, we present the proof of Theorem \ref{main}.

\section{Preliminaries}\label{pre}

\subsection{Generalized quadrangles}\label{GQ}

Suppose that $\mathcal{S}$ is a finite thick generalized quadrangle of order $(s,t)$, where $s,t\ge 2$. Then two distinct points are incident with at most one common line, two distinct lines are incident with at most one common point, and for each point-line pair $(P,\ell)$ that is not incident there is exactly one point on $\ell$ that is collinear with $P$. If $s=t$, $\mathcal{S}$ is said to have order $s$.

\begin{lemma}\label{HD}
Let $\mathcal{S}$ be a finite thick generalized quadrangle of order $(s,t)$ with point set $\mathcal{P}$ and line set $\mathcal{L}$. Then $|\mathcal{P}|=(s+1)(st+1)$, $|\mathcal{L}|=(t+1)(st+1)$, and the following properties hold:
\begin{enumerate}
\item[(i)] (Higman's inequality) $s\leq t^{2}$ and $t\leq s^{2}$;
\item[(ii)] (Divisibility condition) $s+t$ divides $st(s+1)(t+1)$;
\item[(iii)] $\left(\frac{t+1}{s+1}\right)^{3}<|\mathcal{P}|$,
and $\left(\frac{s+1}{t+1}\right)^{3}<|\mathcal{L}|$.
\end{enumerate}
\end{lemma}
\begin{proof}
It suffices to prove (iii), since the other properties are taken from \cite[1.2.1-1.2.3]{Payne}.
By (i), we have  $t+1<(s+1)^2$, i.e., $\frac{t+1}{s+1}<s+1$. It follows that
\[
\left(\frac{t+1}{s+1}\right)^{3}<(s+1)^{2}\cdot\frac{t+1}{s+1}=(1+t)(1+s)<|\mathcal{P}|,
\]
which yields the first inequality of (iii). The other inequality then follows from point-line duality.
\end{proof}

A \emph{grid} with parameters $(s_{1},s_{2})$ is a point-line incidence structure $(\mathcal{P}, \mathcal{L}, \mathrm{I})$ with
$\mathcal{P}=\left\{P_{i, j}: 0 \leqslant i \leqslant s_{1}, 0 \leqslant j \leqslant s_{2}\right\}$, $\mathcal{L}=\left\{\ell_{0}, \ldots, \ell_{s_{1}}, \ell_{0}^{\prime}, \ldots, \ell_{s_{2}}^{\prime}\right\}$
such that $P_{i, j} \mathrm{I} \ell_{k}$ if and only if $i=k$ and $P_{i, j} \mathrm{I} \ell_{k}^{\prime}$ if and only if $j=k$. A \emph{dual} \emph{grid} with parameters $(s_1,s_2)$ is the point-line dual of a grid with parameters $(s_2,s_1)$. A grid with parameters $(s,s)$ is a generalized quadrangle of order $(s,1)$, and a dual grid with parameters $(t,t)$ is also a generalized quadrangle of order $(1,t)$.

\begin{lemma}\cite[2.4.1]{Payne}\label{subquadrangle}
Let $g$ be an automorphism of a finite generalized quadrangle $\mathcal{S}=(\mathcal{P}, \mathcal{L})$ of order $(s,t)$. Let $\mathcal{P}_{g}$ and $\mathcal{L}_g$ be the set of fixed points and fixed lines of $g$ respectively, and let $\mathcal{S}_{g}=(\cP_g,\cL_g)$ be the induced incidence substructure on $\mathcal{P}_{g}\times\mathcal{L}_g$. Then  one of the following holds:
\begin{enumerate}
\item[(0)] $\mathcal{P}_{g}=\mathcal{L}_{g}=\varnothing$,
\item[(1)] $\mathcal{L}_{g}=\varnothing$, $\mathcal{P}_{g}$ is a nonempty set of pairwise noncollinear points,
\item[(1')] $\mathcal{P}_{g}=\varnothing$, $\mathcal{L}_{g}$ is a nonempty set of pairwise nonconcurrent lines,
\item[(2)] $\mathcal{L}_{g}$ is nonempty, and $\mathcal{P}_{g}$ contains a point $P$ that is collinear with each point of $\mathcal{P}_{g}$ and is on each line of $\mathcal{L}_{g}$,
\item[(2')]$\mathcal{P}_{g}$ is nonempty, and $\mathcal{L}_{g}$ contains a line $\ell$ that is concurrent with each line of $\mathcal{L}_{g}$ and contains each point of $\mathcal{P}_{g}$,
\item[(3)] $\mathcal{S}_{g}$ is a grid with parameters $(s_1,s_2)$, $s_{1}<s_{2}$,
\item[(3')] $\mathcal{S}_{g}$ is a dual grid with parameters $(s_1,s_2)$, $s_{1}<s_{2}$,
\item[(4)] $\mathcal{S}_{g}$ is a generalized quadrangle of order $\left(s^{\prime}, t^{\prime}\right)$.
\end{enumerate}
\end{lemma}

\begin{corollary}\label{substructure}
With the same notation as in Lemma \ref{subquadrangle}, we have the following properties:
\begin{enumerate}
\item[(i)]If $|\mathcal{P}_{g}|\ge2$, $|\mathcal{L}_{g}|\ge2$ and $\mathcal{S}_{g}$ admits an automorphism group $H$ that is transitive on both points and lines, then the case (4) of Lemma \ref{subquadrangle} holds.
\item[(ii)]If $|\mathcal{P}_{g}|=|\mathcal{L}_{g}|\ge 2$ and $\mathcal{S}_{g}$ admits an automorphism group $H$ that is transitive on its points, then the case (4) of Lemma \ref{subquadrangle} holds.
\end{enumerate}
\end{corollary}
\begin{proof}
For claim (i), both $\mathcal{P}_{g}$ and $\mathcal{L}_{g}$ have at least two elements, so $\mathcal{S}_{g}$ cannot have  type (0), (1) or (1'). By the transitivity assumption, each point in $\mathcal{P}_{g}$ is on the same number of lines in $\mathcal{L}_{g}$ and each line in $\mathcal{L}_{g}$ contains the same number of points in $\mathcal{P}_{g}$. This further excludes the types (2), (2'), (3), (3') and completes the proof of (i).

It remains to prove (ii). As in the proof of (i), $\mathcal{S}_{g}$ cannot have type (0), (1), (1') or  (2). Write $m:=|\mathcal{P}_{g}|=|\mathcal{L}_{g}|\ge 2$. There is a constant $r$ such that there are exactly $r$ lines in $\mathcal{L}_g$ through each point of $\cP_g$ by the transitivity assumption.
If $\mathcal{S}_{g}$ has type (2'), then $m=|\cL_g|=(r-1)m+1$, i.e., $(2-r)m=1$. This does not hold for $m\ge 2$, so $\mathcal{S}_{g}$ does not have type (2').  If $\mathcal{S}_{g}$ has type (3), then we deduce from $|\mathcal{P}_{g}|=|\mathcal{L}_{g}|$ that $(s_{1}+1)(s_{2}+1)=s_{1}+1+s_{2}+1$. It follows that $s_{1}s_{2}=1$, so $s_{1}=s_{2}=1$: a contradiction. Hence $\mathcal{S}_{g}$ does not have type (3). The type (3') is excluded by the same argument.  This completes the proof.
\end{proof}

\begin{lemma}\label{regular}
If $G$ is a finite group acting regularly on the points of a finite thick generalized quadrangle of order $s$, then it is nonabelian.
\end{lemma}
\begin{proof}
This is a consequence of \cite[Theorem 3.1]{Ghinelli}, cf. the remark following it in \cite{Ghinelli}.
\end{proof}

\begin{lemma}\label{abelian}
Let $\mathcal{S}$ be a finite thick generalized quadrangle of order $(s,t)$. Then there is no abelian group that acts transitively on both the points and the lines of $\mathcal{S}$.
\end{lemma}
\begin{proof}
Suppose to the contrary that an abelian group $G$ acts transitively on both the points and the lines of $\mathcal{S}$. Let $\cP$, $\cL$ be the sets of points and lines of $\cS$ respectively, and fix a point $P$ and a line $\ell$ of $\mathcal{S}$. Take an element $g\in G_P$. Since $G$ is abelian and transitive on the points, $g$ fixes each point of $\cS$, i.e., $g$ is in the kernel of the action of $G$ on the points of $\cS$. It follows that $g$ also fixes each line of $\cS$. In particular, we have $g\in G_\ell$. It follows that $G_P\le G_\ell$. Similarly, we deduce that $G_\ell\le G_P$. We thus have $G_P=G_\ell$, and so $|\cP|=|\cL|$ by the orbit-stabilizer theorem and the transitivity assumption. We deduce that $s=t$ by Lemma \ref{HD}.

By the arguments in the previous paragraph, $G_P$ is contained in the kernel of the action of $G$ on $\cP$, so they are equal. Therefore, the induced action of $G/G_P$ on $\cP$ is regular. The claim then follows from Lemma \ref{regular}.
\end{proof}

\subsection{Maximal subgroups of the almost simple groups with socle $\mathrm{PSL}(2,q)$ }
By using Dickson's classification of the subgroups of $\mathrm{PSL}(2,q)$ \cite{Dickson}, Giudici \cite{Giudici}  determined all the maximal subgroups of  almost simple groups with socle $\mathrm{PSL}(2,q)$, $q\ge 4$. This is also available in \cite[Table 8.1]{Bray}.
\begin{lemma}\label{maximal}
Let $G$ be an almost simple group with socle $X=\mathrm{PSL}(2,q)$, where $q=p^{f} \geqslant4$ for a prime $p$. Let $M$ be a maximal subgroup of $G$ not containing $X$, and set $M_{0}:=M \cap X$. Then either $(G, M, M_0)$ is as in Table \ref{sporadic}, or $M_{0}$ is a maximal subgroup of $X$ as listed in Table \ref{values}.
\end{lemma}
\begin{proof}
By \cite[Theorem 1.1]{Giudici}, either $M_0$ is a maximal subgroup of $X=\mathrm{PSL}(2,q)$, or $(G,M,M_0)$ is as listed in Table \ref{sporadic}. In the last column of Table \ref{sporadic}, we list the index of $M$ in $G$.  The maximal subgroups of $X$ are enumerated in \cite[Theorems 2.1, 2.2]{Giudici}, which we reproduce here in Table \ref{values}. In the second column of Table \ref{values} we list the index of $M_0$ in $X$ for reference.
\end{proof}

\begin{table}
\begin{center}
\caption{The $(G,M,M_0)$ triples with $M$ maximal in $G$ and $M_0$ not maximal in $X$}\label{sporadic}
\begin{tabular}{cccc}
\hline
 $G$ & $M_0$ & $M=N_G(M_0)$ & $[G: M]$ \\ \hline
 $\mathrm{PGL}(2,7)$ & $\mathrm{D}_{6}$ & $\mathrm{D}_{12}$ & 28 \\
 $\mathrm{PGL}(2,7)$ & $\mathrm{D}_{8}$ & $\mathrm{D}_{16}$ & 21 \\
 $\mathrm{PGL}(2,9)$ & $\mathrm{D}_{10}$& $\mathrm{D}_{20}$ & 36 \\
 $\mathrm{PGL}(2,9)$ & $\mathrm{D}_{8}$ & $\mathrm{D}_{16}$ & 45 \\
 $\mathrm{M}_{10}$   & $\mathrm{D}_{10}$& $\mathrm{C}_{5}\rtimes \mathrm{C}_{4}$ & 36 \\
 $\mathrm{M}_{10}$ & $\mathrm{D}_{8}$ & $\mathrm{C}_{8} \rtimes \mathrm{C}_{2}$ & 45 \\
 $\mathrm{P\Gamma L}(2,9)$ & $\mathrm{D}_{10}$ & $\mathrm{C}_{10} \rtimes \mathrm{C}_{4}$ & 36 \\
 $\mathrm{P\Gamma L}(2,9)$ & $\mathrm{D}_{8}$ & $\mathrm{C}_{8} \cdot\mathrm{Aut}\left(\mathrm{C}_{8}\right)$  & 45 \\
 $\mathrm{PGL}(2,11)$ & $\mathrm{D}_{10}$ & $\mathrm{D}_{20}$ & 66 \\
 $\mathrm{PGL}(2,q)$, $q=p \equiv \pm 11,\pm19\pmod {40}$& $\mathrm{A}_{4}$ & $\mathrm{S}_{4}$ & $\frac{1}{24}q(q^{2}-1)$ \\
\hline
\end{tabular}
\end{center}
\end{table}

\begin{table}
\begin{center}
\caption{Maximal subgroups of $X=\mathrm{PSL}(2,q)$ and their indices in $X$}\label{values}
\begin{tabular}{ccccc}\hline
Case & $M_{0}$ & $[X:\,M_{0}]$ & Condition \\\hline
1& $\mathrm{C}_{p}^{f} \rtimes \mathrm{C}_{\frac{q-1}{\gcd(2, q-1)}}$ & $q+1$ & \\
2& $\mathrm{PGL}(2,q_{0})$ & $\frac{q_{0}(q_{0}^{2}+1)}{2}$ & $q=q_{0}^{2}$ odd\\
3& $\mathrm{A}_{5}$ & $\frac{q(q^{2}-1)}{120}$ & $q=p \equiv \pm 1\pmod {10}$ or $q=p^{2}$ and $p \equiv \pm 3\pmod {10}$ \\
4& $\mathrm{A}_{4}$ & $\frac{p(p^{2}-1)}{24}$ & $q=p \equiv \pm 3\pmod 8$ and $p \not \equiv \pm 1\pmod {10}$ \\
5& $\mathrm{S}_{4}$ & $\frac{p(p^{2}-1)}{48}$ & $q=p \equiv \pm 1\pmod 8$ \\
6& $\mathrm{PSL}(2,q_{0})$ & $\frac{q_{0}^{r-1}(q_{0}^{2r}-1)}{q_{0}^{2}-1}$ & $q=q_{0}^{r}$ odd and $r$ is an odd prime \\
7& $\mathrm{PGL}(2,q_{0})$ & $\frac{q_{0}^{r-1}(q_{0}^{2r}-1)}{q_{0}^{2}-1}$ & $q=2^{f}=q_{0}^{r}$, where $r$ is prime and $q_{0} \neq 2$ \\
8& $\mathrm{D}_{2(q-1) / \gcd(2, q-1)}$ & $\frac{q(q+1)}{2}$ & $q\ne 5,7,9,11$\\
9& $\mathrm{D}_{2(q+1) / \gcd(2, q-1)}$ & $\frac{q(q-1)}{2}$ &  $q\ne 7,9$\\
\hline
\end{tabular}
\end{center}
\end{table}

We shall also need the list of all subgroups of $\mathrm{PGL}(2,q)$ with $q$ odd, cf. \cite{Dickson, Huppert, Cameron}.
\begin{lemma}\label{subgroupPGL}(\cite[Theorem 2]{Cameron})
The subgroups of $\mathrm{PGL}(2,q)$ with $q=p^{f}\ge 5$ odd are as follows:
\begin{enumerate}
\item[(i)] $\mathrm{C}_{2}$;
\item[(ii)] $\mathrm{C}_{d}$, where $d \mid q \pm 1$ and $d>2$;
\item[(iii)]$\mathrm{D}_{4}$;
\item[(iv)] $\mathrm{D}_{2d}$, where $d \mid \frac{q \pm 1}{2}$ and $d>2$;
\item[(v)] $\mathrm{D}_{2 d}$, where $(q \pm 1) / d$ is an odd integer and $d>2$;
\item[(vi)]$\mathrm{A}_{4}$, $\mathrm{S}_{4}$ and $\mathrm{A}_{5}$ when $q \equiv \pm 1(\bmod 10)$;
\item[(vii)] $\mathrm{PSL}\left(2, p^{m}\right)$, where $m \mid f$;
\item[(viii)] $\mathrm{PGL}\left(2, p^{m}\right)$, where $m \mid f$;
\item[(ix)] an elementary abelian group of order $p^{m}$, where $m \leq f$;
\item[(x)] a semidirect product of an elementary abelian group of order $p^{m}$ and $\mathrm{C}_{d}$, where $m \leq f$ and $d \mid \gcd(q-1,p^m-1)$.
\end{enumerate}
\end{lemma}


\begin{lemma}\label{fixedpoints}
Let $G$ be a finite transitive permutation group on a set $\Omega$, and choose $\alpha\in\Omega$. For $g\in G$, let $g^G$ be its conjugacy class in $G$ and let $\pi(g)$ be its number of fixed points on $\Omega$. We have
\[\pi(g)=\frac{|\Omega|\cdot|g^G\cap G_\alpha|}{|g^G|}.\]
\end{lemma}
\begin{proof}
The claim follows from double counting the set of pairs $(x,h)$ in $\Omega\times g^G$ such that $h$ fixes $x$, cf. \cite[Lemma 2.5]{Liebeck}. For a character theory version, please refer to \cite[Theorem 5.18]{char}.
\end{proof}

The following results are well-known and can be deduced from \cite[pp.191-193]{Huppert} and \cite[Section 3.1-3.2]{Burness}. We present some more details here for later reference. Let $I_2$ be the identity matrix of order $2$. We write $\textup{diag}(a,b)$ for the $2\times 2$ diagonal matrix whose diagonal entries are $(a,b)$.
\begin{lemma}\label{conjugacy}
Suppose that $X=\mathrm{PSL}(2,q)$, $q=p^f\ge 4$ with $p$ prime. Then $X$ has a single conjugacy class $C$ of involutions and
\[
  \left|C\right|= \begin{cases}q^{2}-1, &\text{ if $q$ is even}, \\ \frac{1}{2}q(q+\epsilon),  &\textup{ if }q\equiv\epsilon\pmod{4}\textup{ with }\epsilon\in\{\pm 1\}.\end{cases}
\]
Moreover, if $g$ is an involution in $X$, then
\[
C_X(g)=\begin{cases}\mathrm{D}_{q-1},&\textup{ if } q\equiv 1\pmod{4},\\
\mathrm{D}_{q+1},&\textup{ if } q\equiv 3\pmod{4},\\
\mathrm{C}_2^f,&\textup{ if $q=2^f$, $f\ge 2$.}\end{cases}
\]
\end{lemma}
\begin{proof}
The first claim is verified in an elementary way in the proof of \cite[Lemma 2.9]{Liu}, see also \cite[Section 4]{InvEven}, \cite{InvOdd} and \cite[Section 3.2]{Burness}.
Here we give an elementary proof of the claim on $C_X(g)$.  Set $A:=\begin{pmatrix}0&1\\-1&0\end{pmatrix}$, so that $A^2=-I_2$.

First suppose that $q$ is odd. Choose an element $i\in\F_{q^2}$ such that $i^2=-1$. If $q\equiv 1\pmod{4}$, then $i$ is in $\F_q^*$. Set $Q=\begin{pmatrix}1&i\\i&1\end{pmatrix}$. Then $Q^{-1}AQ=D$, where $D=\textup{diag}(i,-i)$. Here the entries of $Q$ lie in the field $\F=\F_q[i]$. It is elementary to show that an invertible matrix $B:=\begin{pmatrix}a&b\\c&d\end{pmatrix}$ with entries in $\F$ such that $BD=\lambda DB$ for some constant $\lambda$ if and only if one of the following occurs: (1) $\lambda=1$, $b=c=0$, (2) $\lambda=-1$, $a=d=0$. Take an element $\lambda\in\F$ such that $\lambda^{q+1}=-1$; in particular, we choose $\lambda=i$ if $q\equiv 1\pmod{4}$. Then $N_1=Q\begin{pmatrix}0&\lambda\\-\lambda^{-1}&0\end{pmatrix}Q^{-1}$ lies in $\SL_2(q)$ and $AN_1=-N_1A$ upon direct check.

If $q\equiv 1\pmod{4}$, then up to conjugacy we choose $g$ to be the quotient image of $D$ in $X$. Its centralizer in $X$ is generated by the quotient image of of $\begin{pmatrix}0&i\\i&0\end{pmatrix}$ and those of $\textup{diag}(a,a^{-1})$, $a\in\F_q^*$. This establishes the claim for $q\equiv 1\pmod{4}$.

If $q\equiv 3\pmod{4}$, then $i^q=-i$, and $1,i$ form a basis of $\F_{q^2}$ over $\F_q$. Up to conjugacy we choose $g$ to be the quotient image of $A$ in $X$. Take $h\in C_X(g)$ and let $N$ be a preimage in $\textrm{SL}(2,q)$. Set $B=Q^{-1}NQ$, so that $BD=\lambda DB$ for some constant $\lambda$. If $B$ is diagonal, then $B=\alpha I_2+\beta D$ for some $\alpha,\beta\in\F_{q^2}$ and correspondingly $N=\alpha I_2+\beta A$. Since $N$ has all entries in $\F_q$, we deduce that $\alpha,\beta$ are in $\F_q$.
We have $\det(N)=1$, so $\alpha^2+\beta^2=1$. The set $T:=\{xI_2+yA:\,x,y\in\F_q\}$ forms a field $E$ of order $q^2$ under matrix addition and multiplication, and such $N$'s form a cyclic subgroup $K$ of order $q+1$ of the multiplicative group of $E$. By the second paragraph of this proof, we see that $C_X(g)$ is a dihedral group of order $q+1$ generated by the quotient images of $N_1$ and $K$.

Next suppose that $q$ is even, and up to conjugacy we choose $g$ to be the quotient image of $A$ in $X$. Let $N:=\begin{pmatrix}a&b\\c&d\end{pmatrix}$ be an element of $\textrm{SL}(2,q)$ such that $AN=\lambda NA$ for some constant $\lambda$. By taking determinants on both sides, we deduce that $\lambda=1$. Upon expansion we deduce from $AN=NA$ that $a=d$, $b=c$, i.e., $N$ is in the set $T:=\{xI_2+yA:\,x,y\in\F_q\}$. We deduce that $a+b=1$ from $\det(N)=1$. It is routine to check that such $N$'s form an elementary abelian $2$-group of order $q$. This completes the proof.
\end{proof}

\begin{lemma}\label{conjugacy2}
Suppose that $X=\mathrm{PSL}(2,q)$, $q=p^f$ with $p>5$ prime. Then $X$ has a single conjugacy class $C$ of elements of order $3$ in $X$ and
\begin{align*}
|C|= \begin{cases}q(q-1), &\text{ if }\ q \equiv-1\pmod 3; \\ q(q+1), &\text{ if }\ q \equiv 1\pmod 3.\end{cases}
\end{align*}
Moreover, if $g$ is an element of order $3$ in $X$, then
\[
C_X(g)=\begin{cases}\mathrm{C}_{(q-1)/2},&\textup{ if } q\equiv 1\pmod{3},\\
\mathrm{C}_{(q+1)/2},&\textup{ if } q\equiv 2\pmod{3}.\end{cases}
\]

\end{lemma}
\begin{proof}
Let $g$ be an element of order $3$ in $\PSL(2,q)$. Then $g$ is semisimple and its conjugacy class in $\PSL(2,q)$ is also its conjugacy class in $\PGL(2,q)$ by \cite[Theorem 3.1.12]{Burness}.  We have $\Phi(3,q)=1$ or $2$ according as $q\pmod{3}$ is $1$ or $2$, where $\Phi(3,q)=\min\{i\in\mathbb{N}:\,3\textup{ divides }q^i-1\}$. The number of conjugacy classes as well as the orders of their respective centralizers in $\PGL(2,q)$ can be read off from \cite[Table B.3]{Burness}, so that the first claim follows. The claim on $C_X(g)$ can be proved by the same elementary approach as in the proof of Lemma \ref{conjugacy} and we omit the details.
\end{proof}

\section{Generalized quadrangles with an automorphism group acting transitively on both points and lines}\label{model}

Let $\cS$ be a finite thick generalized quadrangle of order $(s,t)$ with point set $\cP$ and line set $\cL$, and suppose that it admits an automorphism group $G$ that is transitive on both points and lines. Fix a point $\alpha\in\cP$ and a line $\ell\in\cL$.
\begin{lemma}\label{lem_boundM1}
With notation as above,  we have $\frac{|G_{\alpha}|^{4/3}}{|G|^{1/3}}<|G_\ell|<|G_{\alpha}|^{3/4}\cdot|G|^{1/4}$.
\end{lemma}
\begin{proof}
Since $G$ is transitive on both $\cP$ and $\mathcal{L}$, we have $|\cP|=\frac{|G|}{|G_\alpha|}$, $|\cL|=\frac{|G|}{|G_\ell|}$. It follows from Lemma \ref{HD} that  $\frac{t+1}{s+1}=\frac{|G_{\alpha}|}{|G_\ell|}$. By Lemma \ref{HD} (iii), we have
\[
\left(\frac{|G_{\alpha}|}{|G_l|}\right)^3<\frac{|G|}{|G_{\alpha}|},\quad \left(\frac{|G_{\ell}|}{|G_\alpha|}\right)^3<\frac{|G|}{|G_{\ell}|},
\]
from which we deduce the desired inequalities. This completes the proof.
\end{proof}

We define a set $D$ as follows:
\[
D=\{g\in G:\,\alpha^g \textup{ is incident with } \ell\}.
\]
The points on the line $\ell$ are $\alpha^g$ for $g\in D$, and the lines through the point $\alpha$ are $\ell^{g^{-1}}$ for $g\in D$. We thus have $|D|=(s+1)|G_\alpha|=(t+1)|G_\ell|$. The set $D$ is a union of $(G_\alpha,G_l)$-double cosets in $G$, so we have a  decomposition $D=\bigcup\limits_{i=1}^{d} G_{\alpha}h_{i}G_{l}$, where the double cosets $G_{\alpha}h_{i}G_{l}$, $1\le i\le d$, are pairwise distinct. It follows that
\begin{align*}
s+1&=\frac{|D|}{|G_{\alpha}|}=\sum\limits_{i=1}^{d} \frac{|G_{l}|}{|G_{l}\cap h_{i}^{-1}G_{\alpha}h_{i}|},\\
t+1&=\frac{|D|}{|G_{\ell}|}=\sum\limits_{i=1}^{d} \frac{|G_{\alpha}|}{|G_{\alpha}\cap h_{i}G_{l}h_{i}^{-1}|}.
\end{align*}
For $g,h\in G$, the point $\alpha^g$ is incident with the line $\ell^h$ if and only if  $\alpha^{gh^{-1}}$ is incident with $\ell$, i.e., $gh^{-1}\in D$.

\begin{lemma}\label{transitiveaction}
With notation as above, let $g$ be a nonidentity element in $G$. Let $\cP_g$ be the set of fixed points of $g$, and suppose that $\alpha$ is in $\cP_g$. If $[C_{G}(g):\,C_{G}(g)\cap G_{\alpha}]=|\cP_g|$, then $C_{G}(g)$ acts transitively on $\cP_{g}$.
\end{lemma}
\begin{proof}
For $\beta\in\cP_g$ and $x\in C_G(g)$, we have $(\beta^x)^g=(\beta^g)^x=\beta^x$, i.e., $\beta^x$ is in $\cP_g$. Therefore, $C_{G}(g)$ stabilizes $\cP_g$. The point $\alpha$ is in $\cP_g$ and its stabilizer in $C_G(g)$ is $C_{G}(g)\cap G_{\alpha}$. The orbit of $\alpha$ under $C_G(g)$ is thus contained in $\cP_g$, and it has the same size as $\cP_g$ by the condition in the lemma. It follows that $\cP_g$ is exactly the orbit of $\alpha$ under $C_G(g)$, which completes the proof.
\end{proof}

\section{Proof of Theorem \ref{main}}\label{pfmain}
This  section is devoted to the proof of Theorem \ref{main}. Let $\mathcal{S}$ be a finite think generalized quadrangle of order $(s,t)$ with point set $\cP$ and line set $\cL$, and suppose that $G$ is an automorphism group of $\mathcal{S}$ that acts primitively on both $\cP$ and $\cL$. We have $s,t\ge 2$. Assume that $G$ is almost simple with socle $X=\mathrm{PSL}(2,q)$, where $q=p^{f} \geqslant4$ with $p$ prime. Fix a point $\alpha$ and a line $\ell$ of $\mathcal{S}$, and set
\[
  M_{0}:=G_{\alpha}\cap X,\quad M_{1}:=G_{\ell}\cap X.
\]
We have $G_\alpha=N_G(M_0)$, since $G_\alpha$ normalizes $M_0$ and is maximal in $G$. Similarly, $G_\ell=N_G(M_1)$. Since $X$ is normal in $G$, it is transitive on both $\cP$ and $\cL$ by the primitivity assumption. We thus have
\begin{align}
|\cP|=&(s+1)(st+1)=\frac{|X|}{|M_{0}|},\label{point}\\
|\cL|=&(t+1)(st+1)=\frac{|X|}{|M_{1}|}.\label{line}
\end{align}

\begin{lemma}\label{aa}
Neither of $M_{0}$, $M_1$ is isomorphic to $\mathrm{C}_{p}^{f} \rtimes \mathrm{C}_{(q-1) / \operatorname{gcd}(2, q-1)}$.
\end{lemma}
\begin{proof}
By the point-line duality, it suffices to prove the claim for $M_0$.  Suppose to the contrary that $M_0$ is isomorphic to $\mathrm{C}_{p}^{f} \rtimes \mathrm{C}_{(q-1) / \operatorname{gcd}(2, q-1)}$. There is a unique conjugacy class of such subgroups by \cite{Dickson}, see also \cite[Corollary 2.2]{King}. The action of $X=\PSL(2,q)$ on the right cosets of $M_0$ is isomorphic to its natural action on the projective line $\textup{PG}(1,q)$. Therefore, the action of $X$ on $\cP$ is $2$-transitive. This is impossible, since $X$ cannot map collinear points to noncollinear points. This completes the proof.
\end{proof}

\begin{lemma}\label{lem_M0M1_max}
Both $M_0$ and $M_1$ are maximal subgroups of $X$. In particular, the group $X$ acts primitively on both points and lines of $\cS$.
\end{lemma}
\begin{proof}
By the point-line duality, it suffices to prove the claim for $M_0$. Suppose to the contrary that $M_0$ is not maximal in $X$. Then $(G,G_\alpha,M_0)$ is one of the tuples in Table \ref{sporadic} by Lemma \ref{maximal}. By \cite[Lemma 5.1]{Bamberg2012}, $G$ is not one of $\mathrm{PGL}(2,9)$, $\mathrm{P\Gamma L}(2,9)$ or $\mathrm{M}_{10}$. For the three triples with $G=\mathrm{PGL}(2,7)$ or $\mathrm{PGL}(2,11)$ in Table \ref{sporadic}, the numbers of points are $|\cP|=(s+1)(st+1)=28,21,66$ by \eqref{point}, and so $(s,t)=(3,2),(2,3)$ and $(5,2)$ respectively. There are no generalized quadrangles of such orders, since neither of them satisfies the divisibility condition in Lemma \ref{HD}.

It remains to consider the case where  $(G,G_\alpha,M_0)=(\mathrm{PGL}(2,q),\mathrm{S}_{4},\mathrm{A}_{4})$, $q=p\equiv \pm 11,\pm19\pmod {40}$. We claim that $M_1$ is not maximal in $X$. Suppose to  the contrary that $M_1$ is maximal in $X$. Then it is one of the cases in Table \ref{values} by Lemma \ref{maximal}. It is not in Case 1 by Lemma \ref{aa}, and can only be in one of Cases $3$, $8$ or $9$ by the condition on $q$. There are two conjugacy classes of subgroups isomorphic to $\mathrm{A}_5$ in $\PSL(2,q)$, and their normalizers in $G=\PGL(2,q)$ are not maximal in $G$ by \cite[Corollary 2.3]{King}. This excludes Case 3.
For Case 8 where $M_{1}=\mathrm{D}_{p-1}$, we apply Lemma \ref{HD} (iii) to obtain $\left(\frac{p-1}{12}\right)^{3}<\frac{p(p+1)}{2}$. It holds only if $p<867$, and there are no feasible $(s,t)$ pairs for each such prime $p$ by direct check with Magma \cite{Bosma}. This excludes Case 8, and we exclude Case 9 in the same way. Therefore, $M_1$ is not maximal in $X$.

By the claim in the first paragraph of this proof and the point-line duality, we see that $M_1\cong\mathrm{A}_{4}$. It follows from $M_0\cong\mathrm{A}_{4}$ that $s=t$, and so $|\cP|=(s+1)(s^{2}+1)=\frac{1}{24}p(p^{2}-1)$ by \eqref{point}. We claim that $p\equiv 1\pmod{4}$, so that $p\equiv -11,-19\pmod{40}$. If not, then $p$ does not divide $s^2+1$ and so divides $s+1$. It follows that $s\ge p-1$. Then  $|\cP|=(s+1)(s^2+1)\ge p(p^2-2p+2)$ and it is strictly larger than $\frac{1}{24}p(p^{2}-1)$: a contradiction.

Take an involution $g$ in $X$, and let $\cP_g$, $\cL_g$ be the sets of its fixed points and fixed lines respectively. Then $C:=g^X$ is the unique conjugacy  class of involutions in $X$ by Lemma \ref{conjugacy}. Since $M_0\cong \mathrm{A}_4$ has three involutions which form a single conjugacy class $C_0$ of $M_0$, we deduce that $C\cap M_0=C_0$ and $|C_0|=3$. We apply Lemma \ref{fixedpoints} to the transitive action of $X$ on $\cP$ and obtain
\[
|\cP_g|=\frac{|\cP|\cdot|C\cap M_0|}{|C|}=\frac{1}{4}(p-1)>1.
\]
Similarly, we deduce that $|\cL_g|=\frac{1}{4}(p-1)>1$. The involution $g$ fixes more than one point and one line, so we assume without loss of generality that $\alpha,\ell$ are chosen from $\cP_g,\cL_g$ respectively, i.e., $g$ is in both  $M_0$ and $M_1$. We have $|C_X(g)|=p-1$ by Lemma \ref{conjugacy}, $|C_X(g)\cap M_i|=|C_{M_i}(g)|=\frac{1}{3}|\mathrm{A}_4|=4$ for $i=0,1$. It follows from Lemma \ref{transitiveaction} that $C_X(g)$ is transitive on both $\cP_g$ and $\cL_g$. By Corollary \ref{substructure} (i), the fixed structure $\cS_g=(\cP_g,\cL_g)$ is a generalized quadrangle of order $(s',t')$. Since $p\equiv -11$ or $-19\pmod{40}$, $\frac{p-1}{4}$ is odd and at least $7$.
It follows that $s'=t'\ge 2$ and $\cS_g$ is thick. The group $C_X(g)$ is a dihedral group and contains a cyclic subgroup $K$ of order $\frac{1}{4}(p-1)$ by Lemma \ref{conjugacy}. Since $K$ has odd order and $C_X(g)\cap M_i$ has order $4$, we see that $K\cap M_i=1$ for $i=0,1$. It follows that the orbit $\alpha^K$ has size $[K:\,K\cap M_0]=|K|=|\cP_g|$, i.e., $K$ is regular on $\cP_g$. This contradicts Lemma \ref{regular} and completes the proof.
\end{proof}

In view of Lemma \ref{lem_M0M1_max}, we assume without loss of generality that
\[
G=X=\PSL(2,q),\quad G_\alpha=M_{0},\quad G_\ell=M_{1}.
\]
Both of $M_0,M_1$ are maximal subgroups of $X$, so they appear in Table \ref{values}. In the next two subsections, we consider two separate cases according as they have the same Case numbering in Table \ref{values} or not. If they have distinct numberings,  we show that there are no integers $(s,t)$ that satisfy both \eqref{point}, \eqref{line} and the restrictions in Lemma \ref{HD}. If $M_0$ and $M_1$ have the same numberings, we make use of the coset geometry model in Section \ref{model} and the results on the fixed substructure of an automorphism  in Section \ref{pre}. There is one example that arises in the case where $M_0,M_1$ are both in Case $2$ of Table \ref{values}, i.e., $W(2)$. Its full automorphism group has socle $\textup{PSp}(4,2)'\cong \PSL(2,9)$.

\subsection{$M_0,M_1$ have  distinct Case numberings in Table \ref{values}}\label{sneqt}

In this subsection we consider the cases where $M_0,M_1$ have distinct Case numberings in Table \ref{values}. By the point-line duality, we assume without loss of generality that the Case numbering of $M_0$ is smaller than that of $M_1$. By Lemma \ref{aa}, $M_0$ is not in Case 1 of Table \ref{values}. For each given $M_0$, the size of $M_1$ should satisfy the bounds in Lemma \ref{lem_boundM1}. In Table \ref{tab_cand_M1}, we list the possible cases for $M_1$ for a given $M_0$ by considering the restrictions on $q$ and the bounds on $|M_1|$ in Lemma \ref{lem_boundM1}. For instance, if $M_0$, $M_1$ are in Cases 3 and 8 respectively, then the smallest prime power that satisfies the conditions on $q$ in both cases is $19$. The lower bound on $|M_1|$ in Lemma \ref{lem_boundM1} is trivial, and the upper bound yields $\left(\frac{q-1}{60}\right)^{3}<\frac{q(q+1)}{2}$ which holds only if $q<108004$. If $M_0$ is in one of Cases 6-8, then the bounds on $|M_1|$ in Lemma \ref{lem_boundM1} hold trivially and $M_1$ can be any of Cases 7-9.
\begin{table}[]
\begin{center}
\caption{Possible cases for $M_1$ if $M_0$ is in one of Cases 2-5}\label{tab_cand_M1}
\begin{tabular}{ccc}
\hline
Case for $M_0$ & Possible cases for $M_1$ & Condition                                                                                  \\ \hline
2              & 6                        & $q=q_0^2=q_1^r$, $r$ odd prime                                                             \\ \hline
3              & 5                        & $q=p \equiv \pm 1, \pm 9\pmod {40}$                                                        \\
               & 8                        & $q\equiv \pm 1\pmod {10}$, $19\le q<108004$ \\
               & 9                        & $q\equiv \pm 1\pmod {10}$, $19\le q<107996$ \\ \hline
4              & 8                        & $q=p \equiv \pm 3, \pm 27\pmod {40}$, $37\le p<867$                                        \\
               & 9                        & $q=p \equiv \pm 3, \pm 27\pmod {40}$, $37\le p<859$                                        \\ \hline
5              & 8                        & $q=p \equiv \pm 1\pmod 8$, $17\le p<6916$                                                  \\
               & 9                        & $q=p \equiv \pm 1\pmod 8$, $17\le p<6908$                                                  \\ \hline
\end{tabular}
\end{center}
\end{table}

\begin{lemma}
The subgroup $M_{0}$ is not in Case 2 of Table \ref{values}.
\end{lemma}
\begin{proof}
Suppose to the contrary that $M_{0}=\mathrm{PGL}(2,q_{0})$ with $q=q_{0}^{2}$ odd. By Table \ref{tab_cand_M1}, $M_1$ can only be in Case 6 of Table \ref{values}. Suppose that this is the case, i.e., $M_1=\mathrm{PSL}(2,q_{1})$, where $q=q_{1}^{r}$ with $r$ an odd prime. In particular, $q_1$ is a square and $q_0=q_1^{r/2}$. We have $\frac{t+1}{s+1}=\frac{|M_0|}{|M_1|}=\frac{2q_{0}(q_{0}^{2}-1)}{q_{1}(q_{1}^{2}-1)}$ by \eqref{point} and \eqref{line}, and $|\cP|=[X:\,M_0]=\frac{1}{2}q_0(q_0^2+1)$. By Lemma \ref{HD}(iii), we have $\left(\frac{t+1}{s+1}\right)^3<|\cP|$. After plugging in the expressions of $\frac{t+1}{s+1}$ and $|\cP|$, we deduce that
\begin{equation}\label{a1}
16q_1^{r-3}(q_1^r-1)^3<(q_1^2-1)^3(q_1^r+1).
\end{equation}
We have $q_1^r-1>\frac{1}{2}(q_1^r+1)$, since $q_1\ge 9$ and $r\ge 3$. Therefore,
\begin{align*}
16q_1^{r-3}(q_1^r-1)^3&>2q_1^{r-3}(q_1^r+1)^3>2q_1^{3r-3}(q_1^r+1)\\
&\ge 2q_1^{6}(q_1^r+1)>(q_1^2-1)^3(q_1^r+1),
\end{align*}
which contradicts \eqref{a1}. This completes the proof.
\end{proof}

\begin{lemma}
The subgroup $M_{0}$ is not in Case 3 of Table \ref{values}.
\end{lemma}
\begin{proof}
Suppose to the contrary that $M_{0}=\mathrm{A}_{5}$ for $q=p \equiv \pm 1\pmod {10}$ or $q=p^{2}$ with $p \equiv \pm 3\pmod {10}$. By Table \ref{tab_cand_M1}, $M_1$ can be in Cases 5, 8, or 9 of Table \ref{values}, and we have $13\le q<108004$ for Case 8 and $11\le q<107996$ for  Case 9. We verify with computer that there are no feasible $(s,t)$ pairs that satisfy \eqref{point} and \eqref{line} for the latter two cases. It remains to consider Case 5, i.e., $M_1=\mathrm{S}_4$, $q=p\equiv \pm 1\pmod{8}$. In this case, $\frac{t+1}{s+1}=\frac{|M_0|}{|M_1|}=\frac{5}{2}$ by \eqref{point} and \eqref{line}. There exist positive integer $k$ such that $t+1=5k$, $s+1=2k$. Since $s>1$, we have $k>1$. By Lemma \ref{HD}(ii),  we have
\[
0\equiv 7^4st(s+1)(t+1)=7^4\cdot 10k^{2}(2k-1)(5k-1)\equiv -360\pmod{7k-2}.
\]
It follows that $7k-2$ divides $360$, which holds only if $k=2,6$ or $26$. There is no prime $q=p$ that satisfies \eqref{point} and \eqref{line} for such $k$'s. This completes the proof.
\end{proof}

\begin{lemma}
The subgroup $M_{0}$ is not in Case 4 or Case 5 of Table \ref{values}.
\end{lemma}
\begin{proof}
Suppose to the contrary that $M_0=\mathrm{A}_{4}$, where $q=p \equiv \pm 3\pmod 8$ and $q \not \equiv \pm 1\pmod {10}$. By Table \ref{tab_cand_M1}, we have the following candidates for $M_1$: $\mathrm{D}_{p-1}$ with $13\le p<867$, or $\mathrm{D}_{p+1}$ with $p<859$. We verify with computer that there are no feasible $(s,t)$ pairs that satisfy \eqref{point} and \eqref{line} in each case. The case where $M_1=\mathrm{S}_{4}$ is excluded in a similar way. This completes the proof.
\end{proof}

\begin{lemma}\label{f}
The subgroup $M_{0}$ is not in Case 6 of Table \ref{values}.
\end{lemma}
\begin{proof}
Suppose to the contrary that $M_0=\mathrm{PSL}(2,q_{0})$, where $q$ is odd and $q=q_{0}^{r}$ with $r$ an odd prime. We have
\begin{equation}\label{eqn_cPC6}
(1+s)(1+st)=|\cP|=[X:\,M_0]=\frac{q(q^2-1)}{q_0(q_0^2-1)}
\end{equation}
which is odd, so $s$ is even. By the restriction on $q$, $M_1$ can only be Case 8 or Case 9 of Table \ref{values}. The two cases are excluded by the same approach, so we only give the details for Case 8 here. In this case, $M_1=\mathrm{D}_{q-1}$ with $q\ge 13$, $|\cL|=[X:\,M_1]=\frac{1}{2}q(q+1)$, and
\begin{equation}\label{eqn_4C6}
\frac{s+1}{t+1}=\frac{|M_1|}{|M_0|}=\frac{2(q-1)}{q_{0}(q_{0}^{2}-1)}.
\end{equation}
We plug them into the inequality $\left(\frac{s+1}{t+1}\right)^3<|\cL|$ from Lemma \ref{HD} (iii) and deduce that $16(q-1)^3<q(q+1)q_0^3(q_0^2-1)^3$. Since $(q-1)^3=q(q+1)(q-4)+(7q-1)$, it follows that $16(q-4)<q_0^3(q_0^2-1)^3$. It does not hold when $r\ge 11$, so $r\le 7$.

Since $q$ and $r$ are odd, the greatest common divisor of the numerator and denominator of the right hand side of \eqref{eqn_4C6} is $q_0-1$. Therefore, there is an integer $k$ such that $s+1=\frac{q_{0}^{r}-1}{q_{0}-1}\cdot k$ and $t+1=\frac{1}{2}kq_{0}(q_{0}+1)$. Since $s$ is even, $k$ is odd. It is then routine to show that $k$ is relatively prime to $1+st$.  We deduce from \eqref{eqn_cPC6} that
\begin{equation}\label{eqn_tmp1}
 k(1+st)=\frac{q_0^r+1}{q_0+1}q_0^{r-1},
\end{equation}
We deduce from $1+st>1+s$ that $k^2<\frac{(q_0^r+1)(q_0-1)}{(q_0^r-1)(q_0+1)}q_0^{r-1}<q_0^{r-1}$, so $k<q_0^{(r-1)/2}$. We deduce from \eqref{eqn_tmp1} that $q_0^{r-1}$ divides either $k$ or $1+st$. Since $k<q_0^{(r-1)/2}$, $q_0^{r-1}$ divides $1+st$.
There is an integer $u$ such that
$uk=\frac{q_0^r+1}{q_0+1}$, $1+st=q_0^{r-1}u$. Write $k=b+aq_0$ with $0\le b\le q_0-1$ and $a\in\mathbb{Z}$. Taking modulo $q_0^2$, we obtain
\[
s+1\equiv b+(a+b)q_0,\; 2(t+1)\equiv bq_0, \;2+2st\equiv 0\pmod{q_0^2}.
\]
We determine $s,2t\pmod{q_0^2}$ from the first two equations and plug them into the third to obtain $2+(b-1+(a+b)q_0)(bq_0-2)\equiv 0\pmod{q_0^2}$. Taking modulo $q_0$, we obtain $b=2$. The former equation then yields $a\equiv -1\pmod{q_0}$. We thus have $k\equiv 2-q_0\pmod{q_0^2}$. In particular, we have $k\ge q_0^2-q_0+2$. If $r=3$, then it contradicts the fact $k<q_0^{(r-1)/2}=q_0$. If $r=5$, then we deduce from $k<q_0^2$ that $k=q_0^2-q_0+2$, but $k$ does not divide $\frac{q_0^5+1}{q_0+1}$ in this case: a contradiction.

It remains to consider the case $r=7$. Write $k=2-q_0+cq_0^2$ with $c\in\mathbb{Z}$. By considering $1+s$, $1+t$, $2(1+st)$ modulo $q_0^3$, we obtain $c\equiv\frac{q_0+1}{2}\pmod{q_0}$ by a similar procedure. It follows from $k<q_0^3$ that $k=\frac{q_0+1}{2}q_0^2-q_0+2$. We have $\frac{q_0^7+1}{q_0+1}\equiv 33q_0^2-49q_0+57\pmod{k}$, and $33q_0^2-49q_0+57$ is nonzero and smaller than $k$ if $q_0>61$. Hence $k$ does not divide $\frac{q_0^7+1}{q_0+1}$ when $q_0>61$. We check by computer that $k$ does not divide $\frac{q_0^7+1}{q_0+1}$ when $13\le q_0\le 61$ either. This completes the proof.
\end{proof}

\begin{lemma}\label{g}
The subgroup $M_{0}$ is not in Case 7 of Table \ref{values}.
\end{lemma}
\begin{proof}
Suppose to the contrary that $M_0$ is in Case 7 of Table \ref{values}. Then $M_0=\mathrm{PGL}(2,q_{0})$, where $q=q_{0}^{r}$ is even, $r$ is a prime and $q_{0} \neq 2$.  Since we have assumed that $M_1$ has larger Case numbering in Table \ref{values}, $M_1$ can only be in Case 8 or Case 9 of Table \ref{values}. The case $r\ge 3$ are dealt with in the same way as in the proof of Lemma \ref{f}, so we omit the details here. Suppose that $r=2$, and we only give the details for Case 8 since the other case is similar. Write $q_0=2^n$, $n\ge 2$. We have $|\cP|=(1+s)(1+st)=q_0(q_0^2+1)$, $|\cL|=(1+t)(1+st)=\frac{1}{2}q_0^2(q_0^2+1)$, so $\frac{t+1}{s+1}=2^{n-1}$. Write $k=s+1$, so that $t+1=2^{n-1}k$. By Lemma \ref{HD} (i) we have $t+1<(s+1)^2$, so $k>2^{n-1}$. It follows that $s\ge 2^{n-1}$, $t\ge 2^{2n-2}+2^{n-1}-1$. We thus have
\[
2^n(2^{2n}+1)=|\cP|\ge(2^{n-1}+1)(1+2^{n-1}(2^{2n-2}+2^{n-1}-1)).
\]
It holds only if $n=2,3$. We check with computer that there are no $(s,t)$ pairs that satisfy \eqref{point} and \eqref{line} for each such $n$. This completes the proof.
\end{proof}

\begin{lemma}\label{h}
The subgroup $M_{0}$ is not in Case 8 of Table \ref{values}.
\end{lemma}
\begin{proof}
Suppose to the contrary that $M_{0}=\mathrm{D}_{2(q-1) / \gcd(2, q-1)}$. Since we have assumed that $M_1$ has larger Case numbering in Table \ref{values}, we have $M_1=\mathrm{D}_{2(q+1) / \gcd(2, q-1)}$.
We have $|\cP|=(1+s)(1+st)=\frac{1}{2}q(q+1)$, $|\cL|=(1+t)(1+st)=\frac{1}{2}q(q-1)$ by \eqref{point} and \eqref{line}. It follows that  $(s-t)(st+1)=q$, so $s-t$  and $(st+1)$ are both powers of $p$. Write $s-t=p^{h}$ and $st+1=p^{f-h}$, where $q=p^{f}$. Since $1+st>1$, we have $h<f$. We thus have $1+s=\frac{1}{2}p^h(q+1)$, $1+t=\frac{1}{2}p^h(q-1)$, and so $s+t=p^{f+h}-2$. By Lemma \ref{HD}(ii) $s+t$ divides $st(st+1)$, so  $p^{f+h}-2$ divides $(p^{f-h}-1)p^{f-h}$.
\begin{enumerate}
\item[(1)] If $p$ is odd, then  $p^{f+h}-2$ is relatively prime to $p$ and so divides $p^{f-h}-1$. Since $h<f$, $p^{f-h}-1$ is positive and so $p^{f-h}-1\ge p^{f+h}-2$. This inequality holds only if $h=0$, but then $p^{f}-2$ does not divides $p^{f}-1$ by the fact $p^f\ge 13$.
\item[(2)] If $p=2$, then $2^{f+h-1}-1$ divides $2^{f-h}-1$. The latter is positive by the fact $h<f$. It follows that $2^{f-h}-1\ge 2^{f+h-1}-1$, which implies that $h=0$. Then $2^{f-1}-1$ divides $2^f-1$, and we deduce that $2^f=4$. Then $s+t=2^f-2=2$, which contradicts the fact that $s,t\ge 2$.
\end{enumerate}
This completes the proof.
\end{proof}

\subsection{$M_0,M_1$ have the same Case numberings in Table \ref{values}}\label{seqt}

In this subsection, we suppose that $M_0,M_1$ have the same Case numbering in Table \ref{values}. We start by showing that $M_0$ and $M_1$ are isomorphic groups.

\begin{lemma}\label{PSL_67}
The groups $M_0$ and $M_1$ are isomorphic.
\end{lemma}
\begin{proof}
Suppose to the contrary that $M_0$ and $M_1$ are not isomorphic. This can happen only if they are both in Case 6 or Case 7 in Table \ref{values}. In both cases, $M_0\cong\mathrm{PSL}(2,q_{0})$ and $M_1\cong\mathrm{PSL}(2,q_{1})$, where $q=q_{0}^{r_{0}}=q_{1}^{r_{1}}$ and $r_{0}$, $r_{1}$ are distinct primes. Moreover, if $q$ is odd, then both $r_0,r_1$ are odd; if $q$ is even, then $q_0,q_1>2$.
There exists a prime power $m$ such that $q=m^{r_0r_1}$, and so $q_{0}=m^{r_{1}}$ and $q_{1}=m^{r_{0}}$.

There is a unique conjugacy class of subgroups isomorphic to $\PSL(2,m)$ in $\PSL(2,m^u)$ by \cite{Dickson} or \cite[Theorem 2.1(o)(p)]{King}, where $u$ is odd if $m$ is odd. For each $i=0,1$, take a subgroup $H_{i}$ isomorphic to $\mathrm{PSL}(2,m)$ in $M_{i}$. There exists $h\in\mathrm{PSL}(2,q)$ such that $H_{0}=H_{1}^{h}$. The stabilizer of $\ell^h$ in $X$ is $M_1^h$ which contains $H_0$. By choosing the line $\ell^h$ instead of $\ell$ in the first place, we assume without loss of generality that $M_0,M_1$ both contain a subgroup $H$ isomorphic to $\mathrm{PSL}(2,m)$. Let $g$ be an involution in $H$. It is clear that $g^X\cap Y=g^Y$ for $Y\in\{M_0,M_1\}$ since each of $M_0,M_1$ has a single conjugacy class of involutions by Lemma \ref{conjugacy}. Let $\cS_g=(\cP_g,\cL_g)$ be the fixed structure of $g$.

First suppose that $q$ is odd. We only give details for the case $m\equiv 3\pmod{4}$ here, and the case $m\equiv 1\pmod{4}$ is dealt with similarly. In this case, $q,q_0,q_1$ are all congruent to $3$ modulo $4$. We have $C_X(g)=\mathrm{D}_{q+1}$, $C_X(g)\cap M_0=\mathrm{D}_{q_0+1}$, $C_X(g)\cap M_1=\mathrm{D}_{q_1+1}$ by Lemma \ref{conjugacy}. We thus have  $|\mathcal{P}_g|=\frac{q+1}{q_{0}+1}$, $|\mathcal{L}_{g}|=\frac{q+1}{q_{1}+1}$ by Lemma \ref{fixedpoints}. By Lemma \ref{transitiveaction}, $C_X(g)$ acts transitively on $\mathcal{P}_g$ and $\mathcal{L}_{g}$. By Corollary \ref{substructure}, $\mathcal{S}_g$ is a generalized quadrangle of order $(s',t')$. Since both $|\cP_g|, |\cL_g|$ are odd integers, $s'$, $t'$ are both even and $\cS_g$ is thick. Let $K$ be the unique cyclic subgroup of index $2$ in $C_X(g)$, so that it intersects $M_i$ in a cyclic subgroup of order dividing $\frac{q_i+1}{2}$, $i=0,1$. Since the orbit $\alpha^K$ has size $[K:\,K\cap M_0]$ which is at most $|\cP_g|$, we deduce that $|K\cap M_0|=\frac{q_0+1}{2}$ and $K$ is transitive on $\cP_g$. Similarly, we deduce that $K$ is transitive on $\cL_g$. This is impossible by Lemma \ref{abelian}.

Next suppose that $q$ is even. Then $C_X(g)$, $C_X(g)\cap M_0$, $C_X(g)\cap M_1$ are elementary abelian $2$-groups of order $q,q_0,q_1$ respectively by Lemma \ref{conjugacy}. We thus have  $|\mathcal{P}_g|=q/q_0$, $|\mathcal{L}_{g}|=q/q_1$ by Lemma \ref{fixedpoints}. By Lemma \ref{transitiveaction}, $C_X(g)$ acts transitively on $\mathcal{P}_g$ and $\mathcal{L}_{g}$. By Corollary \ref{substructure}, $\mathcal{S}_g$ is a generalized quadrangle of order $(s',t')$. If $t'=1$, then $2(s'+1)=m^{r_0(r_1-1)}$, $(s'+1)^2=m^{r_1(r_0-1)}$. It follows that $m^{2r_0(r_1-1)}=4m^{r_1(r_0-1)}$. With $m=2^e$, we deduce that $2+r_1(r_0-1)e=2r_0(r_1-1)e$. It follows that $e$ divides $2$. If $e=2$, then  $r_0(r_1-2)+(r_1-1)=0$ which never holds.
We thus have $e=1$. Then $(r_1-2)(r_0+1)=0$, and so $r_1=2$. By Lemma \ref{HD} (iii), we have $|M_1|^4<|M_0|^3\cdot|X|$ which simplifies to $q_1^2(q_1^2-1)^3<60^3(q_1^2+1)$. It does not hold for $q_1=2^{r_{0}}>4$. Therefore, $t'>1$. Dually, we have $s'>1$, so $\cS_g$ is thick. This contradicts Lemma \ref{abelian} and completes the proof.
\end{proof}

By Lemma \ref{PSL_67}, we assume that $M_0$ and $M_1$ are isomorphic in the sequel. In particular, we have $s=t$. By the arguments in Section \ref{model}, there exist elements $h_1,\cdots, h_d$ of $X$ such that the $M_0h_iM_1$'s are pairwise distinct and
\begin{equation}\label{doublecoset}
s+1=\sum\limits_{i=1}^{d} \frac{|M_{1}|}{|M_{1}\cap h_{i}^{-1}M_{0}h_{i}|}.
\end{equation}
\begin{lemma}\label{lem_M0Case2}
If $M_{0}$ is in Case 2 of Table \ref{values}, then $q=9$ and  $\mathcal{S}$ is $W(2)$.
\end{lemma}
\begin{proof}
Suppose that $M_{0}=\mathrm{PGL}(2,q_{0})$ with $q=q_{0}^{2}$ odd. Write $q_{0}=p^{n}$ with $p$ prime. By \eqref{point}, we have
\begin{equation}\label{bb}
(s+1)(s^{2}+1)=\frac{1}{2}q_{0}(q_{0}^{2}+1).
\end{equation}
If $q_0=5$, there is no integer solution in $s$. If $q_0=3$, then $s=2$. By \cite[5.2.3]{Payne}, up to isomorphism $W(2)$ is the only generalized quadrangle of order 2. The group $\mathrm{PSL}(2,9)\cong\mathrm{PSp}(4,2)'$ acts transitively on the points of $W(2)$, and the stabilizer of a point is a maximal subgroup isomorphic to $\mathrm{PGL}(2,3)$. We assume that $q_{0}\geq7$ in the sequel. It follows  from \eqref{bb} that $s<q_0$. Each summand $\frac{|M_{1}|}{|M_{1}\cap h_{i}^{-1}M_{0}h_{i}|}$ in \eqref{doublecoset} is thus no more than $q_0$, i.e., $[M_{1}:\,M_{1}\cap M_{0}^{h_i}]\le q_0$ for $1\le i\le d$. We examine the subgroups of $M_1\cong\mathrm{PGL}(2,q_{0})$ as listed in Lemma \ref{subgroupPGL}, and deduce that a subgroup of index at most $q_0$ is either  $\mathrm{PSL}\left(2, q_{0}\right)$ or $\mathrm{PGL}\left(2, q_{0}\right)$. If $M_{1}\cap M_{0}^{h_i}=\mathrm{PSL}\left(2, q_{0}\right)$, then it has index $2$ and  is thus normal in both $M_{1}$ and $M_{0}^{h_i}$. That is, $M_{1}$ and $M_{0}^{h_i}$ are both in $N_{X}(\mathrm{PSL}(2, q_0))=\mathrm{PGL}(2,q_0)$. We deduce that $M_{1}=M_{0}^{h_i}=\mathrm{PGL}(2,q_0)$ by comparing sizes: a contradiction to $M_{1}\cap M_{0}^{h_i}=\mathrm{PSL}\left(2, q_{0}\right)$. Therefore, we must have $M_{1}\cap M_{0}^{h_i}=\mathrm{PGL}\left(2, q_{0}\right)$, i.e., $M_{1}=M_{0}^{h_i}$,  for each $i$. It follows that $d=s+1$, and $h_ih_1^{-1}$ is in $N_{X}(M_0)=M_0$. Therefore, $M_0h_1=M_0h_i$, for $1\le i\le s+1$. This contradicts the fact that the $M_0h_iM_1$'s are distinct double cosets.
\end{proof}

\begin{lemma}\label{gg}
The subgroup $M_{0}$ is not in Case 7 of Table \ref{values}.
\end{lemma}
\begin{proof}
Suppose to the contrary that $M_{0}=\mathrm{PGL}(2,q_{0})$, where $q=2^{f}=q_{0}^{r}$ with $r$ prime and $q_{0}\neq2$. From \eqref{point}, we have
\begin{equation}\label{ggp}
(s+1)(s^{2}+1)=\frac{(q_{0}^{2r}-1)}{q_{0}^{2}-1}q_{0}^{r-1}.
\end{equation}
The right hand side is even, so $s$ is odd. We have  $s^{2}+1\equiv 2\pmod {4}$, so $\frac{1}{2}q_{0}^{r-1}$ divides  $s+1$. There is an odd integer $t$ such that $s+1=\frac{1}{2}q_{0}^{r-1}t$. The left hand side of \eqref{ggp} is larger than or smaller than the right hand side according as $t\ge 3$ or $t=1$ upon inspection. Therefore, \eqref{ggp} has no integer solution in $s$. This completes the proof.
\end{proof}

\begin{lemma}\label{hh}
The subgroup $M_{0}$ is not in Case 8 of Table \ref{values}.
\end{lemma}

\begin{proof}

Suppose to the contrary that $M_0=\mathrm{D}_{2(q-1)/\gcd(2,q-1)}$, where $q\ge 13$ if $q$ is odd. From \eqref{point}, we have
\begin{equation}\label{hhp}
(s+1)(s^{2}+1)=\frac{1}{2} q(q+1).
\end{equation}

First suppose that $q$ is odd. We deduce from \eqref{hhp} that $s<q-1$. In particular, $q$ does not divide $s+1$. Since $\gcd(s+1,s^{2}+1)$ is at most $2$, we deduce from \eqref{hhp} that $q$ divides $s^2+1$. It then follows that $s+1$ divides $\frac{1}{2}(q+1)$. Write $\frac{1}{2}(q+1)=k(s+1)$ for some integer $k$. Then $q=2ks+2k-1$, and \eqref{hhp} yields $s^{2}+1=(2ks+2k-1)k$.  As a quadratic equation in $s$, it has an integer solution. Hence  its discriminant $\Delta:=4k^{4}+8k^{2}-4k-4$ is an even square. Since $(2k^{2})^{2}\leq \Delta<(2k^{2}+2)^{2}$, we deduce that $\Delta=(2k^{2})^{2}$. This holds only if $k=1$. Solving the quadratic equation in $s$, we obtain $s=2$. It follows that $q=5$, contradicting the condition $q\geq13$.

Next suppose that $q$ is even, and write $q=2^f$, $f\ge 2$. If $q=4$, then \eqref{hhp} has no integer solution in $s$. Assume that $f\ge 3$ in the sequel. The right hand side of \eqref{hhp} is even, so $s$ is odd. Since $s^{2}+1\equiv 2\pmod {4}$, we deduce from \eqref{hhp} that $s+1=2^{f-2}a$ for some divisor $a$ of $2^{f}+1$. Then \eqref{hhp} reduces to $(2^{2f-5}a^{2}-2^{f-2}a+1)a=2^{f}+1$. It is elementary to show that the left hand side is an increasing function in $a$ for $a\ge 1$, and is strictly larger than $2^f+1$ when $a\ge3$. Therefore, we must have $a=1$. However,  $2^{2f-5}-2^{f-2}+1=2^{f}+1$ holds for no integer $f\ge 3$. This completes the proof.
\end{proof}

\begin{lemma}\label{ii}
The subgroup $M_{0}$ is not in Case 9 of Table \ref{values}.
\end{lemma}
\begin{proof}
Suppose to the contrary that $M_{0}=\mathrm{D}_{2(q+1)/\gcd(2,q-1)}$, $q\ne 7,9$. From \eqref{point}, we have \[(1+s)(1+s^2)=\frac{1}{2}q(q-1).\] By similar arguments to those in the proof of Lemma \ref{hh}, we deduce that $s=9$, $q=41$. Then $|\cP|=820$. Let $g$ be an involution in $M_{0}$, and let $\cS_g=(\cP_g,\cL_g)$ be its fixed structure.  By Lemma \ref{conjugacy}, $g^X$ consists of all the involutions of $X$, $|g^X|=861$ and $|C_X(g)|=40$. The dihedral group $\mathrm{D}_{42}$ has $21$ involutions, so $|M_i\cap g^X|=21$ for $i=0,1$. We deduce from Lemma \ref{fixedpoints} that $|\cP_g|=|\cL_g|=20$. We assume without loss of generality that we have chosen $\alpha,\ell$ from $\cP_g,\cL_g$ respectively, so that $g$ is in both $M_0,M_1$. The centralizer of $g$ in $M_i$ is $C_{M_i}(g)=\langle g\rangle$ for $i=0,1$, so we deduce from Lemma \ref{transitiveaction} that $C_X(g)$ is transitive on both $\cP_g$ and $\cL_g$. By Corollary \ref{substructure}, $\cS_{g}$ is a generalized quadrangle of order $(s',t')$. Since $|\cP_g|=|\cL_g|=20$, we have $s'=t'$ and $(1+s')(1+s'^2)=20$. There is no such integer $s'$: a contradiction. This completes the proof.
\end{proof}

In the sequel, we consider Cases 3, 4, 5, 6 in Table \ref{values} for $M_0$. We shall make use of the results in Section \ref{model} to exclude those cases. This will avoid solving the Diophantine equations involving $(s,q)$ arising from \eqref{point}. We start with a simple observation.

\begin{lemma}\label{lem_A5A4_p}
If $M_0=\mathrm{A}_5$ or $\mathrm{A}_4$, then $p\equiv 1\pmod{4}$.
\end{lemma}

\begin{proof}
If $p=3$, then we must have $M_0=\mathrm{A}_5$, $q=9$. By \eqref{point},
$(s+1)(s^{2}+1)=6$ which has no integer solution in $s$. Similarly, we have $p\neq5$. Hence we assume that $p\ge 7$ in the sequel.
We prove the claim for $M_0=\mathrm{A}_5$, and the other case is similar. By \eqref{point}, we have
\[
(s+1)(s^{2}+1)=\frac{1}{120}q(q^{2}-1).
\]
It follows that $s<q-1$, and so $q$ does not divide $s+1$. Since $\gcd(s+1,s^2+1)$ is at most $2$ and $p\ge 7$, we see that $q$ divides $s^2+1$. This implies that $-1$ is a square modulo $p$, so $p\equiv 1\pmod{4}$ as desired.
\end{proof}

Suppose that $M_0$ is one of the groups in Table \ref{transitive}, where the Case column refers to its numbering in Table \ref{values}. There is a unique conjugacy class of elements of the specified order $r$ in the third column in $X$ by Lemma \ref{conjugacy},  and we fix such an element $g$ that is contained in $M_0$. Let $\cS_g=(\cP_g,\cL_g)$ be the fixed structure of $g$. The size of $g^X$ and the structure of $C_X(g)$ are available in Lemmas \ref{conjugacy} and \ref{conjugacy2}. Then $g^X\cap M_i$ consists of all the elements of order $r$ in $M_i$ which form a single conjugacy class of $M_i$, where $i=0,1$. We are thus able to calculate $|g^X\cap M_i|$ and $C_{M_i}(g)$ for $i=0,1$ in each case either by Magma \cite{Bosma} or Lemma \ref{conjugacy}. By Lemma \ref{fixedpoints}, we calculate $|\cP_g|$ and $|\cL_g|$ which turn out to be equal. The conditions on $q$ in the last column of Table \ref{transitive} arise from Lemmas \ref{conjugacy}, \ref{conjugacy2} or \ref{lem_A5A4_p}.

\begin{table}
\begin{center}
\caption{Information about elements of specified orders in the maximal subgroup $M_{i}$ of $X$}\label{transitive}
\begin{tabular}{ccccccccc}\hline
Case &$M_{i}$ &  $o(g)$ & $|g^X|$ & $|g^X\cap M_{i}|$ &  $C_{X}(g)$ & $C_{M_{i}}(g)$ & $|\mathcal{P}_g|,|\mathcal{L}_{g}|$  & Condition\\ \hline

3 &$\mathrm{A}_{5}$ & 2 & $\frac{q(q+1)}{2}$ & 15 &  $\mathrm{D}_{q-1}$ & $\mathrm{C}_{2}\times\mathrm{C}_{2}$ &$\frac{q-1}{4}$ &  $p\equiv 1\pmod{4}$\\

4& $\mathrm{A}_{4}$ & 2 & $\frac{q(q+1)}{2}$ & 3  & $\mathrm{D}_{p-1}$ & $\mathrm{C}_{2}\times\mathrm{C}_{2}$ & $\frac{p-1}{4}$ &  $p\equiv 1\pmod{4}$\\

5& $\mathrm{S}_{4}$ & 3 & $p(p+1)$ & 8  & $\mathrm{C}_{\frac{p-1}{2}}$ & $\mathrm{C}_{3}$ & $\frac{p-1}{6}$ &  $p\equiv 1\pmod{3}$\\

& $\mathrm{S}_{4}$ & 3 &  $p(p-1)$ & 8 & $\mathrm{C}_{\frac{p+1}{2}}$ & $\mathrm{C}_{3}$ & $\frac{p+1}{6}$ &  $p\equiv 2\pmod{3}$\\

6& $\mathrm{PSL}(2,q_{0})$ & 2 & $\frac{q(q-1)}{2}$ & $\frac{q_{0}(q_{0}-1)}{2}$ & $\mathrm{D}_{q+1}$ & $\mathrm{D}_{q_{0}+1}$ & $\frac{q+1}{q_{0}+1}$ & $q_{0}\equiv 3\pmod{4}$\\

& $\mathrm{PSL}(2,q_{0})$ & 2 & $\frac{q(q+1)}{2}$ & $\frac{q_{0}(q_{0}+1)}{2}$ &  $\mathrm{D}_{q-1}$ & $\mathrm{D}_{q_{0}-1}$ & $\frac{q-1}{q_{0}-1}$ &$q_{0}\equiv 1\pmod{4}$\\ \hline
\end{tabular}
\end{center}
\end{table}

The group $C_X(g)$ stabilizes both $\cP_g$ and $\mathcal{L}_{g}$, and both have sizes greater than $1$ by direct check. We assume without loss of generality that $\alpha,\ell$ are chosen from $\cP_g,\cL_g$ respectively.  By Lemma \ref{transitiveaction},  we deduce that $C_X(g)$ is transitive on both $\cP_g$ and $\cL_g$ in each case. By Corollary \ref{substructure}, $\mathcal{S}_g$ is a generalized quadrangle of order $(s',t')$. Since $\cP_g$ and $\cL_g$ have the same size, we have $s'=t'$. If $p=23$ and $M_0\cong \mathrm{S}_{4}$, then $|\cP_g|=4$ and $|\cP|=11\cdot 23=(1+s)(1+s^2)$ has no integer solution in $s$. It is routine to check that $|\cP_g|\ne 4$ in the other cases, so $\cS_g$ is thick, i.e., $s'\ge 2$.

\begin{lemma}\label{regulargroup}
The subgroup $M_{0}$ is not in one of Cases 3, 4, 5, 6 in Table \ref{values}.
\end{lemma}
\begin{proof}
We continue with the arguments preceding this lemma. Let $K$ be a cyclic subgroup of $C_X(g)$ of the largest possible order, which is unique in each case. We first establish the facts about $K$ in Table \ref{tab_K}. It suffices to determine $|K\cap M_i|$ for $i=0,1$ in each case, since the remaining information in Table \ref{tab_K} follows easily. We only give details for $K\cap M_0$, since $K\cap M_1$ is determined similarly. In the Cases 3, 4, 5, $K$ contains $g$ as is clear in the proof of Lemma \ref{conjugacy}. Since $K\cap M_0$ is cyclic and $C_{M_0}(g)$ is elementary abelian $r$-group with $r=o(g)$, we deduce that $K\cap M_0=\langle g\rangle$ as desired in the Cases 3, 4, 5. We observe that a cyclic subgroup of $\mathrm{D}_{4m}$ has order dividing $2m$, so for Case 6 the subgroup $K\cap M_0$ has order dividing $\frac{1}{2}|C_{M_0}(g)|$. Therefore, $|\alpha^K|$ is a multiple of  $\frac{2\cdot |K|}{|C_{M_0}(g)|}=|\cP_g|$. Since $\alpha^K$ is contained in $\cP_g$, we deduce that $|\alpha^K|=|\cP_g|$ and so $|K\cap M_0|=\frac{1}{2}|C_{M_0}(g)|$ as desired.

\begin{table}
\begin{center}
\caption{Information about the cyclic subgroup $K$}\label{tab_K}
\begin{tabular}{cccccc}\hline
Case & $C_X(g)$ &  $K$ & $K\cap M_{i}$ &  $[K:\,K\cap M_i]$  & Condition\\ \hline
3  & $\mathrm{D}_{q-1}$& $\mathrm{C}_{\frac{q-1}{2}}$ &$\langle g\rangle$ & $\frac{q-1}{4}$ & $p\equiv 1\pmod{4}$\\
4  & $\mathrm{D}_{p-1}$& $\mathrm{C}_{\frac{p-1}{2}}$ &$\langle g\rangle$ & $\frac{p-1}{4}$ & $p\equiv 1\pmod{4}$\\

5  & $\mathrm{C}_{\frac{p-1}{2}}$ & $\mathrm{C}_{\frac{p-1}{2}}$ & $\langle g\rangle$ &  $\frac{p-1}{6}$ &  $p\equiv 1\pmod{3}$\\

   & $\mathrm{C}_{\frac{p+1}{2}}$ & $\mathrm{C}_{\frac{p+1}{2}}$ & $\langle g\rangle$ &   $\frac{p+1}{6}$ &  $p\equiv 2\pmod{3}$\\

6  & $\mathrm{D}_{q+1}$ & $\mathrm{C}_{\frac{q+1}{2}}$ & $\mathrm{C}_{\frac{q_0+1}{2}}$ &
$\frac{q+1}{q_{0}+1}$ & $q_{0}\equiv 3\pmod{4}$\\

   & $\mathrm{D}_{q-1}$ & $\mathrm{C}_{\frac{q-1}{2}}$ & $\mathrm{C}_{\frac{q_0-1}{2}}$ &
$\frac{q-1}{q_{0}-1}$ & $q_{0}\equiv 1\pmod{4}$\\  \hline
\end{tabular}
\end{center}
\end{table}
We are now ready to establish the claim. From Table \ref{tab_K}, we deduce that $|\alpha^K|=[K:\,K\cap M_0]=|\cP_g|$, $|\ell^K|=[K:\,K\cap M_1]=|\cL_g|$, so that $K$ is transitive on both $\cP_g$ and $\cL_g$ in each case. This contradicts Lemma \ref{abelian} and completes the proof.
\end{proof}

We now summarize the results that we have proved so far in this section. In Lemmas \ref{aa} and \ref{lem_M0M1_max}, we have excluded the cases where $M_0:=X\cap G_\alpha$ or $M_1:=X\cap G_l$ is not maximal in $X=\mathrm{PSL}(2,q)$. Assume that $M_0$ and $M_1$ are both maximal in $X$, so that $X$ is also primitive on both points and lines. We assume without loss of generality that $G=X$. In Section \ref{sneqt} we have handled the cases where $M_0,M_1$ have different Case numbering in Table \ref{values}, and in Section \ref{seqt} we have handled the cases where $M_0$ and $M_1$ have the same Case numbering in the table. Putting together, this completes the proof of Theorem \ref{main}.

\vspace*{10pt}

\noindent\textbf{Acknowledgements.}  The authors thank John Bamberg for pointing out an error in an earlier version. They also thank the anonymous referees for their valuable comments and suggestions. This work was supported by National Natural Science Foundation of China (Grant Numbers 12171428 and 12225110).

\begin{center}
	\scriptsize
	\setlength{\bibsep}{0.5ex}  
		\linespread{0.5}
	\bibliographystyle{plain}

\end{center}

\end{document}